\documentclass[12pt, reqno]{amsart}
\usepackage{amsmath, amsthm, amscd, amsfonts, amssymb, graphicx, color}
\usepackage[bookmarksnumbered, colorlinks, plainpages]{hyperref}
\hypersetup{colorlinks=true,linkcolor=red, anchorcolor=green, citecolor=cyan, urlcolor=red, filecolor=magenta, pdftoolbar=true}

\newtheorem{theorem}{Theorem}[section]
\newtheorem{lemma}[theorem]{Lemma}
\newtheorem*{theorem*}{Question}
\newtheorem{proposition}[theorem]{Proposition}
\newtheorem{corollary}[theorem]{Corollary}
\theoremstyle{definition}

\theoremstyle{remark}
\newtheorem{remark}[theorem]{Remark}
\numberwithin{equation}{section}

\begin{document}

\setcounter{page}{1}

\title[On the convexity of the Berezin range]{On the  convexity of the Berezin range of Composition operators and related questions}
\author[Athul Augustine, M. Garayev \MakeLowercase{and} P. Shankar]{Athul Augustine, M. Garayev \MakeLowercase{and} P. Shankar}

\address{Athul Augustine, Department of Mathematics, Cochin University of Science And Technology,  Ernakulam, Kerala - 682022, India. }
\email{\textcolor[rgb]{0.00,0.00,0.84}{athulaugus@gmail.com, athulaugus@cusat.ac.in}}

\address{M. Garayev, Department of Mathematics, College of Science , King Saud University, P.OBox 2455Riyadh 11451, Saudi Arabia }
\email{\textcolor[rgb]{0.00,0.00,0.84}{mgarayev@ksu.edu.sa}}

\address{P. Shankar, Department of Mathematics, Cochin University of Science And Technology,  Ernakulam, Kerala - 682022, India.}
\email{\textcolor[rgb]{0.00,0.00,0.84}{shankarsupy@gmail.com, shankarsupy@cusat.ac.in}}




\begin{abstract}
The Berezin range of a bounded operator $T$ acting on a reproducing kernel Hilbert space $\mathcal{H}$ is the set $B(T)$ := $\{\langle T\hat{k}_{x},\hat{k}_{x} \rangle_{\mathcal{H}} : x \in X\}$, where $\hat{k}_{x}$ is the normalized reproducing kernel for $\mathcal{H}$ at $x \in X$. In general, the Berezin range of an operator is not convex. Primarily, we focus on characterizing the convexity of the Berezin range for a class of composition operators acting on the Fock space on $\mathbb{C}$ and the Dirichlet space of the unit disc $\mathbb{D}$. We prove an analogue of the elliptic range theorem for the unitarily equivalent Berezin range of an operator on a two-dimensional reproducing kernel Hilbert space and characterize the convexity of the unitarily equivalent Berezin range for a bounded operator $T$ on a reproducing kernel Hilbert space $\mathcal{H}$.

\end{abstract}
\maketitle

\section{Introduction}

The reproducing kernel Hilbert space (RKHS)  \cite{paulsen2016introduction} is a Hilbert space $\mathcal{H} = \mathcal{H}(X)$ of
complex-valued functions on some set $X$ such that the evaluation functional $E_{y}: \mathcal{H}\rightarrow \mathbb{C}$, defined by $E_{y}(f) = f(y)$ is bounded on $\mathcal{H}$ for each point $y \in X$. If $\mathcal{H}$ is a reproducing kernel Hilbert space on $X$ and $x \in X$, then by the Riesz representation theorem in functional analysis, there exists a unique function $k_x \in \mathcal{H}$ such that $E_{x}(f)=f(x) = \langle f, k_x \rangle_{\mathcal{H}}$ for all $f \in \mathcal{H}$. The element $k_x$ is called the \textit{reproducing kernel} at $x$ and the \textit{normalized reproducing} at $x$ denoted by $\hat{k}_x =  k_x/||k_x||_{\mathcal{H}}$. Let $B(\mathcal{H})$ denote the set  of all  bounded linear operators  on the Hilbert space $\mathcal{H}$. For $x \in X$, the Berezin transform of an operator $T \in B(\mathcal{H})$ at $x$ is defined as $\widetilde{T}(x) := \langle T \hat{k}_x,\hat{k}_x\rangle_\mathcal{H}.$ Then the Berezin range of T is defined as
$$\text{Ber}(T) := \{\langle T \hat{k}_x,\hat{k}_x\rangle_\mathcal{H}: x\in X\}$$
where $\hat{k}_x$ is the normalized reproducing kernel for $\mathcal{H}$ at $x \in X$.

F.A Berezin \cite{berezin1972covariant} introduced the Berezin transform of an operator on a reproducing kernel Hilbert space as a tool in quantization. The Berezin transform of an operator provides essential information about the operator. In particular, it is known that on the most familiar reproducing kernel Hilbert spaces, including the Hardy space, the Bergman space, the Dirichlet space and the Fock spaces, the Berezin transform uniquely determines the operator, i.e., if $T_1,T_2 \in B(\mathcal{H})$, then $T_1=T_2$ if and only if $\widetilde{T}_1 = \widetilde{T}_2$. 

Consider a bounded linear operator $T$ on a Hilbert space $\mathcal{H}$. Then the \textit{numerical range} of $T$ is defined as 
$$W(T):= \{\langle Tu,u\rangle : \|u\| = 1 \}.$$
By the Toeplitz-Hausdorff theorem \cite{toeplitz}
the numerical range of a linear operator on a Hilbert space is always convex. It is easy to observe that the Berezin range of an operator in a reproducing kernel Hilbert space is a subset of its numerical range. Given that the convexity of the numerical range is arguably its most important characteristic, we are prompted to pose the primary question explored in this paper:

\begin{theorem*}
 Given a
class of concrete operators acting on a reproducing kernel Hilbert space $\mathcal{H}$, what can be said about the convexity of the Berezin range of these operators?
\end{theorem*}

The study on the geometry of the Berezin range was initiated by Karaev \cite{karaev2013reproducing}. The convexity of the Berezin range for finite-dimensional
matrices, multiplication operators, and a class of composition operators acting
on the Hardy space of the unit disc were characterized in \cite{cowen22}. In \cite{augustine2023composition}, characterization of the convexity of the Berezin range was extended to infinite dimensional matrices, composition operators on the Hardy space for the general case of the
elliptic symbol, and composition operators on the Bergman space for the general case of elliptic symbol and the Blaschke factor.

This paper is the natural continuation of the results in \cite{augustine2023composition, cowen22, karaev2013reproducing}.  A relation between the Berezin range $\text{Ber}(T)$ of an operator $T$ and its numerical range was explored in \cite{karaev2013reproducing}. So another important question of our interest is as follows. Given an operator $T$ on an RKHS $\mathcal{H}$, are there any properties of $W(T)$ that can be deduced from $\text{Ber}(T)$?

Consider a bounded linear operator $T$ on a Hilbert space $\mathcal{H}$. Then the \textit{unitarily equivalent Berezin range}  $\mathcal{B}(T)$ is defined as the set of all Berezin transforms of all operators that are unitarily equivalent to $T$. That is, if $\mathcal{A}$ is the collection of all operators that are unitarily equivalent to $T$, then  
	$$\mathcal{B}(T) = \bigcup_{A\in \mathcal{A}} \text{Ber}(A).$$ 
This set was introduced in \cite{nordgren1994boundary} to characterize the compactness of an operator on a standard functional Hilbert space. This paper establishes an analogue of the elliptic range theorem for the unitarily equivalent Berezin range on a two-dimensional Hilbert space. In the proof of this theorem, we explicitly provide the unitary matrix that maps the normalized reproducing kernel to each element in the numerical range of the operator.

This paper is organized as follows. In Section 2, we characterize the convexity
of the Berezin range for a class of composition operators on the Fock space. In Section 3, we characterize the convexity of the Berezin range for a class of  composition operators on the Dirichlet space of the unit disc. In Section 4, we prove an analogous result of the
elliptic range theorem for the unitarily equivalent Berezin range.  As a corollary, we demonstrate that the unitarily equivalent Berezin range equals the numerical range for any bounded linear operator in a reproducing kernel Hilbert space.

\section{Composition operator on Fock space}
Let $d\nu (z)$ denote the Lebesgue measure on the complex plane $\mathbb{C}$. Consider $\mathbb{C}$ with Gaussian measure 
$$d \mu (z) = \frac{1}{\pi}e^{-|z|^2}d\nu (z).$$ 
The Fock space $F^2(\mathbb{C})$ \cite{zhu2012analysis} consists of the Gaussian square integrable entire functions on $\mathbb{C}$. $F^2(\mathbb{C})$ is a closed subspace of $L^2(\mathbb{C},d\mu)$. The Fock space $F^2(\mathbb{C})$ is a reproducing kernel Hilbert space with the following inner product inherited from $L^2(\mathbb{C},d\mu)$:
$$\langle f,g \rangle = \int_\mathbb{C} f(z)\overline{g(z)}d \mu (z).$$
 The reproducing kernel of $F^2(\mathbb{C})$ is given by $k_w(z)=K(z,w) = e^{z\bar{w}} ~~~\text{where}~~~ z,w\in\mathbb{C}.$

Let $\phi$ be a complex-valued function $\phi : X\longrightarrow X$. A composition operator $C_\phi$ acting on a space of functions is defined by 
$$ C_\phi f := f \circ \phi.$$

Acting on the Fock space $F^2(\mathbb{C})$, these operators have the Berezin transform
\begin{center}
  	\begin{equation*}
  		\begin{split}
  			\widetilde{C_{\phi}}(z) &=\langle C_{\phi}\hat{k}_{z},\hat{k}_{z}\rangle\\
  			&=\frac{1}{||k_z||^2}\langle C_{\phi}{k_{z}},{k_{z}}\rangle\\
  			&=\frac{1}{e^{|z|^2}} k_{z}(\phi(z)).
  		\end{split}
  	\end{equation*}
  \end{center}
  
Carswell, MacCluer, and Schuster \cite{carswell2003composition} characterized the bounded composition operators acting on the Fock space $F^2_\alpha$ of holomorphic functions on $\mathbb{C}^n$. Since we focus on the entire functions on $\mathbb{C}$, we state their result for the case $n=1$.
\begin{theorem} \cite{carswell2003composition} \label{carswell}
The composition operator $C_\phi$ is bounded on the Fock space $F^2(\mathbb{C})$ only when $\phi(z)= \zeta z + a$ with $|\zeta|\leq 1.$ Conversely, suppose that $\phi(z)= \zeta z + a$. If $|\zeta|=1$ then $C_\phi$ is bounded on $F^2(\mathbb{C})$ if and only if $a=0$. If $|\zeta|<1$ then  $C_\phi$ acts compactly on $F^2(\mathbb{C})$. 
\end{theorem}

For $z \in \mathbb{C}$ and $\zeta \in \overline{\mathbb{D}}$ and $a=0$ let $\phi(z) = \zeta z$ and consider the composition operator $C_{\phi}f = f\circ \phi.$ Acting on $F^2(\mathbb{C})$, we have
  \begin{center}
  	\begin{equation}\label{2.1}
  		\begin{split}
  			\widetilde{C_{\phi}}(z) &=\langle C_{\phi}\hat{k}_{z},\hat{k}_{z}\rangle\\
  			&=\frac{1}{e^{|z|^2}} k_{z}(\phi(z))\\
  			&=\frac{1}{e^{|z|^2}} k_{z}(\zeta z)\\
  			&=\frac{1}{e^{|z|^2}} e^{\zeta |z|^2}\\
  			&= e^{(\zeta - 1) |z|^2}
  		\end{split}
  	\end{equation}
  \end{center}
 The Berezin range of these operators is not always convex, as shown in Figure \ref{fig1}. Here, we try to find the values of $\zeta$ for which $B(C_\phi)$ is convex.

  \begin{figure}[h]
  	\caption{$B(C_\phi)$ on $F^2(\mathbb{C})$ for $\zeta=0.5e^{i\frac{\pi}{3}}$(apparently not convex).}
  	\includegraphics[scale=0.5]{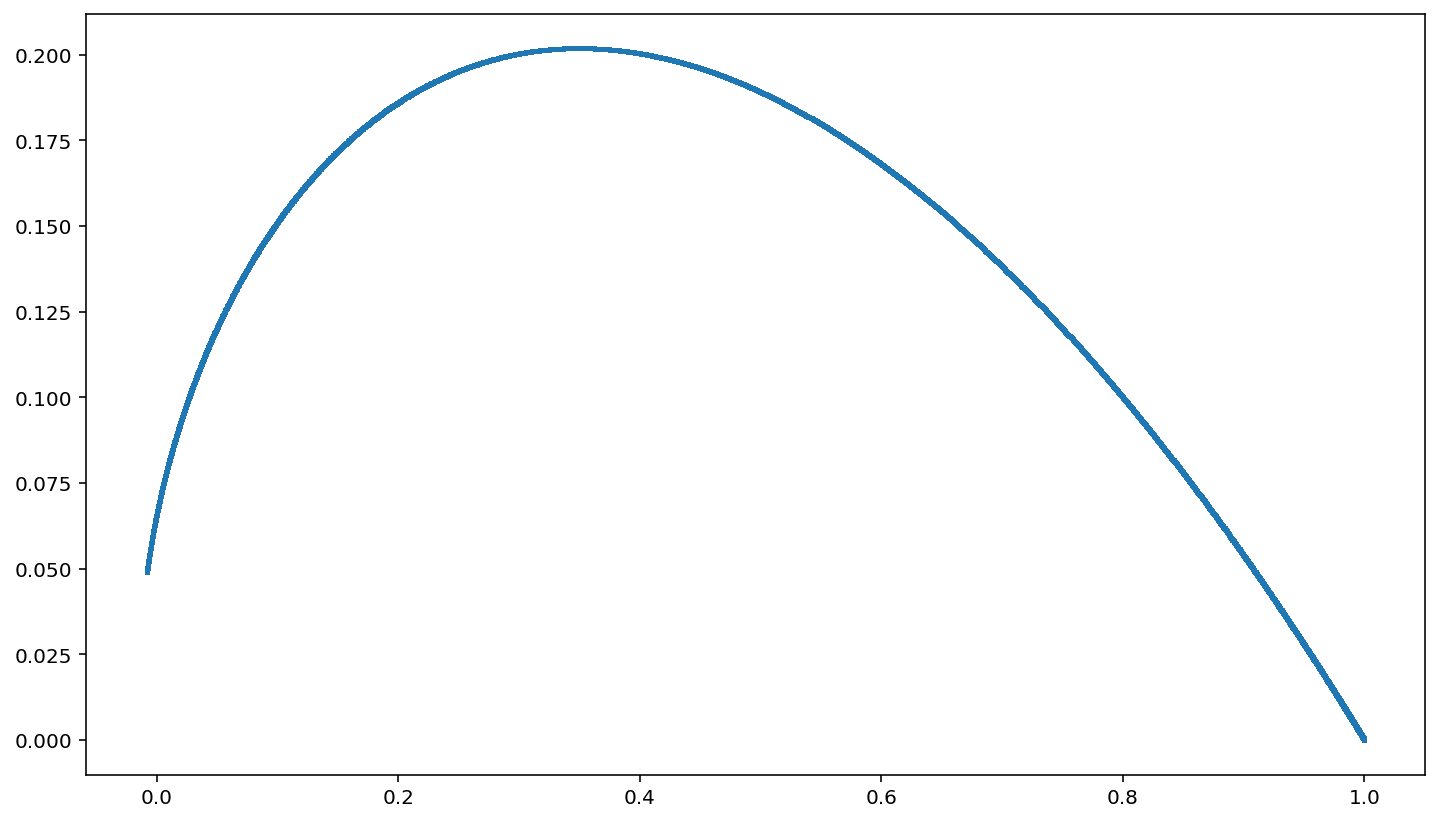}
  	\label{fig1}
  	 \end{figure}

\begin{theorem}\label{elliptic}                                                                                                                                                                                                                                                                                                                                                                                                                                                                                                                                                                                                                                                                                                                                                                                                                                                                                                                                                                                                                                                                                                                                                                                                                                                                                                                                                                                                                                                                                                                                                                                                                                                                                                                                                                                                                                                                                                                                                                                                                                                                                                                                                                                                                                                                                                                                                                                                                                                                                                                                                                     
Let $z\in \mathbb{C}$,  $\zeta \in \overline{\mathbb{D}}$ and $\phi(z)=\zeta z$. Then the Berezin range of $C_\phi$ acting on $F^2(\mathbb{C})$ is convex if and only if $\zeta \in [-1,1]$.
\end{theorem}
\begin{proof}
Let us prove the backward implication first. Suppose that $\zeta=1$. Then  $\phi(z)=z$. Equation (\ref{2.1}) implies that 
$$\widetilde{C_{\phi}}(z) = e^{(\zeta - 1) |z|^2} = e^0 = 1.$$  
Thus, $\text{Ber}(C_\phi)=\{1\}$, which is a singleton set and is a convex set in $\mathbb{C}.$ 

Now suppose that $\zeta \in [-1,1).$ Let $z=re^{i\theta}$. Then for $0\leq r < \infty$,
$$\text{Ber}(C_\phi)= \{e^{(\zeta - 1) r^2} : r\in[0,\infty)\}=(0,1]$$ 
 which is also convex in $\mathbb{C}$.

Conversely, we need to prove that if $\text{Ber}(C_\phi)$ is convex then $\zeta \in [-1,1]$.  It is enough to prove that if the imaginary part of $\zeta$ is non-zero, then $\text{Ber}(C_\phi)$ is not convex. Assume that the imaginary part of $\zeta \neq 0.$ Let $\zeta-1 = a+ib$ for $b\neq 0$. Then
$$\widetilde{C_{\phi}}(re^{i\theta}) = e^{(a+ib) r^2}=e^{ar^2}e^{ibr^2} \qquad r\in [0,\infty),$$
which is a function independent of $\theta.$ Therefore, $\text{Ber}(C_\phi)$ is just a path in 
$\mathbb{C}.$ A path in $\mathbb{C}$ can be convex if and only if it is a point or a line segment.

 For $a=0,~~~ \widetilde{C_{\phi}}(re^{i\theta}) = e^{ibr^2}$. Since $r\in [0,\infty)$, $\text{Ber}(C_\phi)$ is nothing but the unit circle, which is not convex. 
 
 Now for $a\neq 0$, conversely assume that $\text{Ber}(C_\phi)$ is convex. We have
$$\widetilde{C_{\phi}}(re^{i\theta}) =e^{ar^2}e^{ibr^2} \qquad r\in [0,\infty).$$
Our assumption, $\text{Ber}(C_\phi)$ is convex, will imply that it is a point or a line segment. For $r=0$, it is easy to observe that $\widetilde{C_{\phi}}(re^{i\theta}) = 1$ and for $r=1$, $\widetilde{C_{\phi}}(re^{i\theta}) =e^{a}e^{ib}$.  Since $a$ and $b$ are non-zero real numbers, $1$ and $e^{a}e^{ib}$ are two distinct points in the complex plane. Thus, $\text{Ber}(C_\phi)$ is a line segment. Then, every point in $\text{Ber}(C_\phi)$ is collinear. Consider the points $1$ and $e^{a}e^{ib}$ in $\text{Ber}(C_\phi)$ corresponding to the values of $r=0$ and $r=1$ respectively, and an arbitraty point in $\text{Ber}(C_\phi)$, $e^{a\rho^2}e^{ib\rho^2}$ where $\rho \in [0,\infty)$. In the $\mathbb{R}^2$ plane, these three points can  be viewed as $(1,0), (e^a cosb,e^a sinb)$ and $(e^{a\rho^2} cos(b\rho^2), e^{a\rho^2} sin(b\rho^2))$ respectively. Then, from collinearity, we have
 $$1(e^a sinb-e^{a\rho^2} sin(b\rho^2)) + e^a cosb (e^{a\rho^2} sin(b\rho^2) - 0 ) + e^{a\rho^2} cos(b\rho^2)(0 - e^a sinb) = 0.$$
 Rearranging and simplifying the above equation, we get
 $$e^a sinb = e^{a\rho^2}\left[sin(b\rho^2) - e^asin(b\rho^2+b)\right] \quad \forall \rho \in [0,\infty)$$
 In particular, put $\rho = \sqrt{\pi/b}$ in the above equation. Then
 \begin{center}
 \begin{equation*}
 \begin{split}
 e^a sinb &= e^{\frac{a\pi}{b}}\left[sin(\pi) - e^asin(\pi+b)\right]\\
 &=e^{\frac{a\pi}{b}}\left[e^asin b\right].
 \end{split}
 \end{equation*}
 \end{center}
 By cancelling $e^a sinb$ both sides we get
 $$ e^{\frac{a\pi}{b}} =1 = e^0$$
 This implies that $a\pi = 0$, thus $a=0$. This contradicts our assumption $a\neq 0$. So $\text{Ber}(C_\phi)$ is not convex. Therefore, if the imaginary part of $\zeta$ is non-zero, then $\text{Ber}(C_\phi)$ is not convex.
\end{proof}
By Theorem \ref{carswell}, for $\phi(z) = \zeta z + a$ and $|\zeta| = 1$, $C_\phi$ is bounded on $F^2(\mathbb{C})$ if and only if $a=0$. So if we take $\zeta \in \mathbb{T}$, $C_\phi$ is bounded on $F^2(\mathbb{C})$ if and only if  $\phi (z) =\zeta z$. In that case, the convexity of the Berezin range can be characterized completely as follows:
\begin{corollary}\label{fock}
Let $z\in \mathbb{C}$,  $\zeta \in \mathbb{T}$ and $\phi(z)=\zeta z$. Then the Berezin range of $C_\phi$ acting on $F^2(\mathbb{C})$ is convex if and only if $\zeta = -1 $ or $\zeta = 1 $.
\end{corollary}

Now consider the automorphisms of the unit disc $\phi(z) = \frac{az+b}{\overline{b}z+\bar{a}}$ where $a,b\in \mathbb{C}$ and $|a|^2 - |b|^2=1$. For the composition operator $C_{\phi} $
acting on $F^2(\mathbb{C})$, we have
$$\widetilde{C_{\phi}}(z)=\frac{1}{e^{|z|^2}} k_{z}(\phi(z)).$$

 The Berezin range of these operators is not always convex, as shown in Figure \ref{fig2}. Here, we try to find the values of $\alpha$ for which $B(C_\phi)$ is convex.

  \begin{figure}[h]
    	\caption{$B(C_\phi)$ on $F^2(\mathbb{C})$ for $a=e^{i\frac{\pi}{12}}$ and $b=0$(apparently not convex).}
  	\includegraphics[scale=0.5]{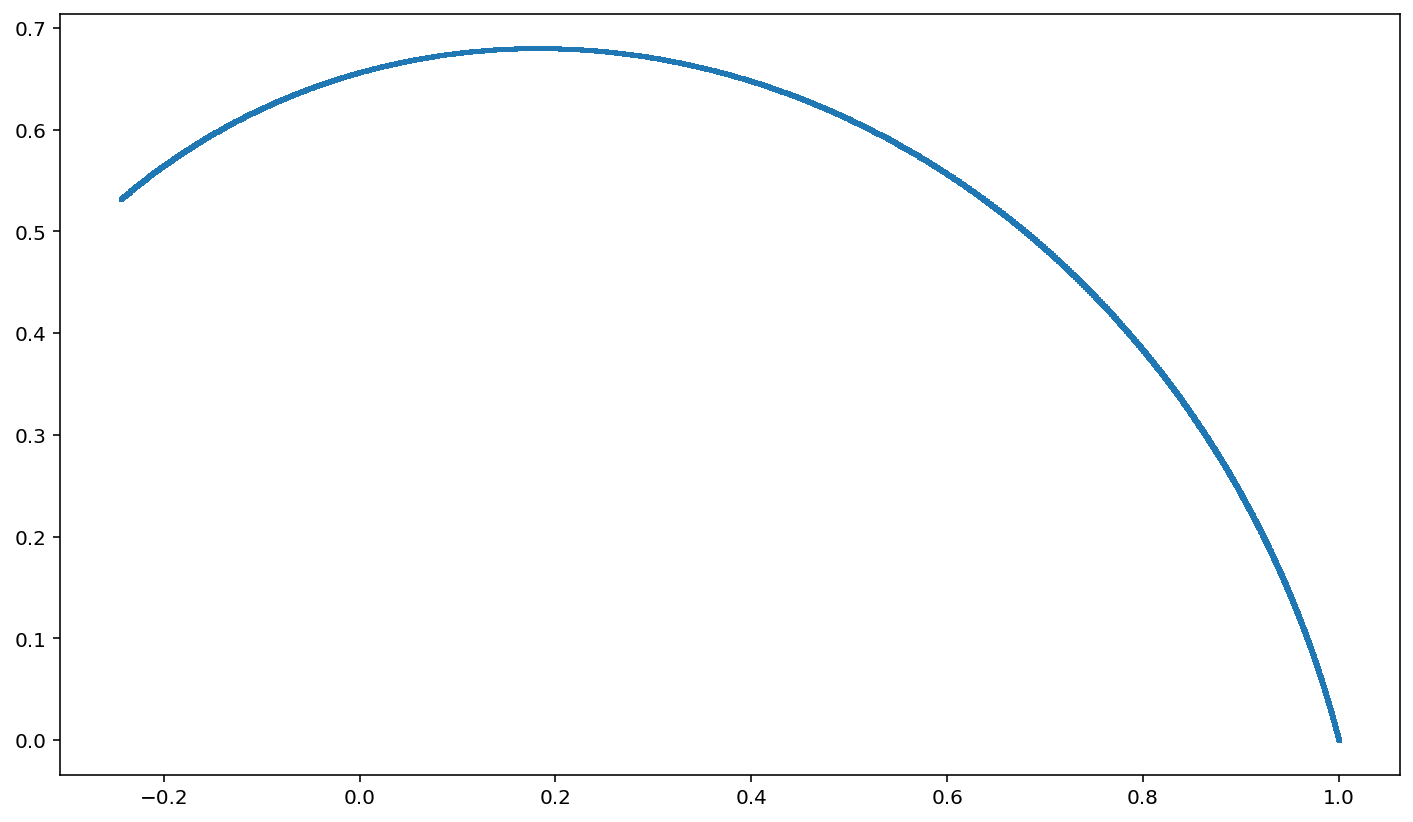}
  	\label{fig2}
  	 \end{figure}

The following remark is a consequence of Corollary \ref{fock}.
  \begin{remark}
    	For $b=0$, $\text{Ber}(C_\phi)$ is convex if and only if $a=1,-1,i,-i$.
  \end{remark}
Now for the case $\phi(z) = \zeta z + a$ with $|\zeta| < 1$ and $a\neq 0$, the Berezin range of the corresponding
composition operator is not always convex, as shown in Figure \ref{fig3}.
\begin{figure}[h]
    	\caption{$\text{Ber}(C_\phi)$ on $F^2(\mathbb{C})$ for $\zeta=0.5$ and $a=1$(left, apparently not convex) and for $\zeta=0.5$ and $a=10$(right, apparently convex).}
  	\includegraphics[scale=0.25]{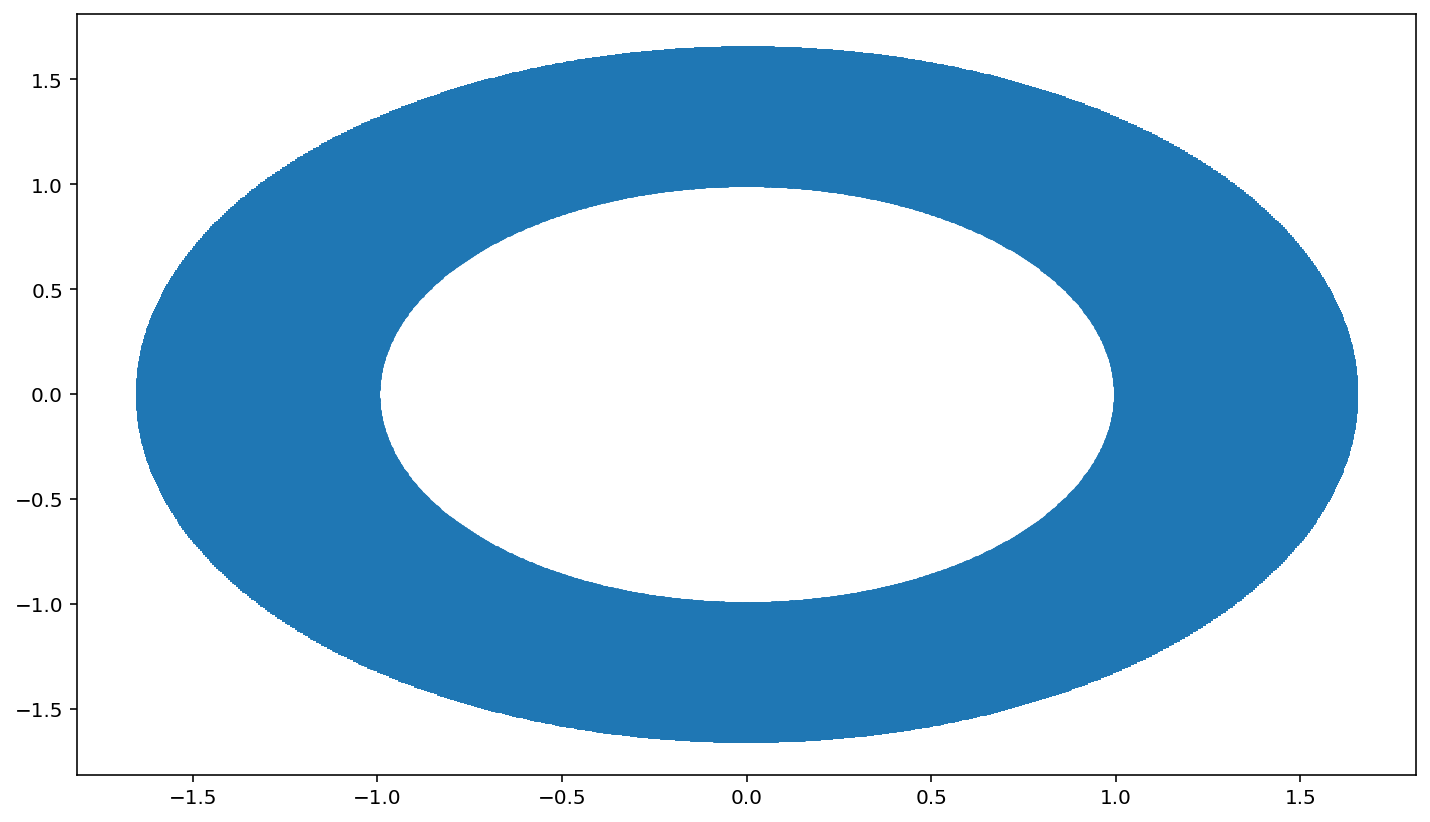}
  	\includegraphics[scale=0.25]{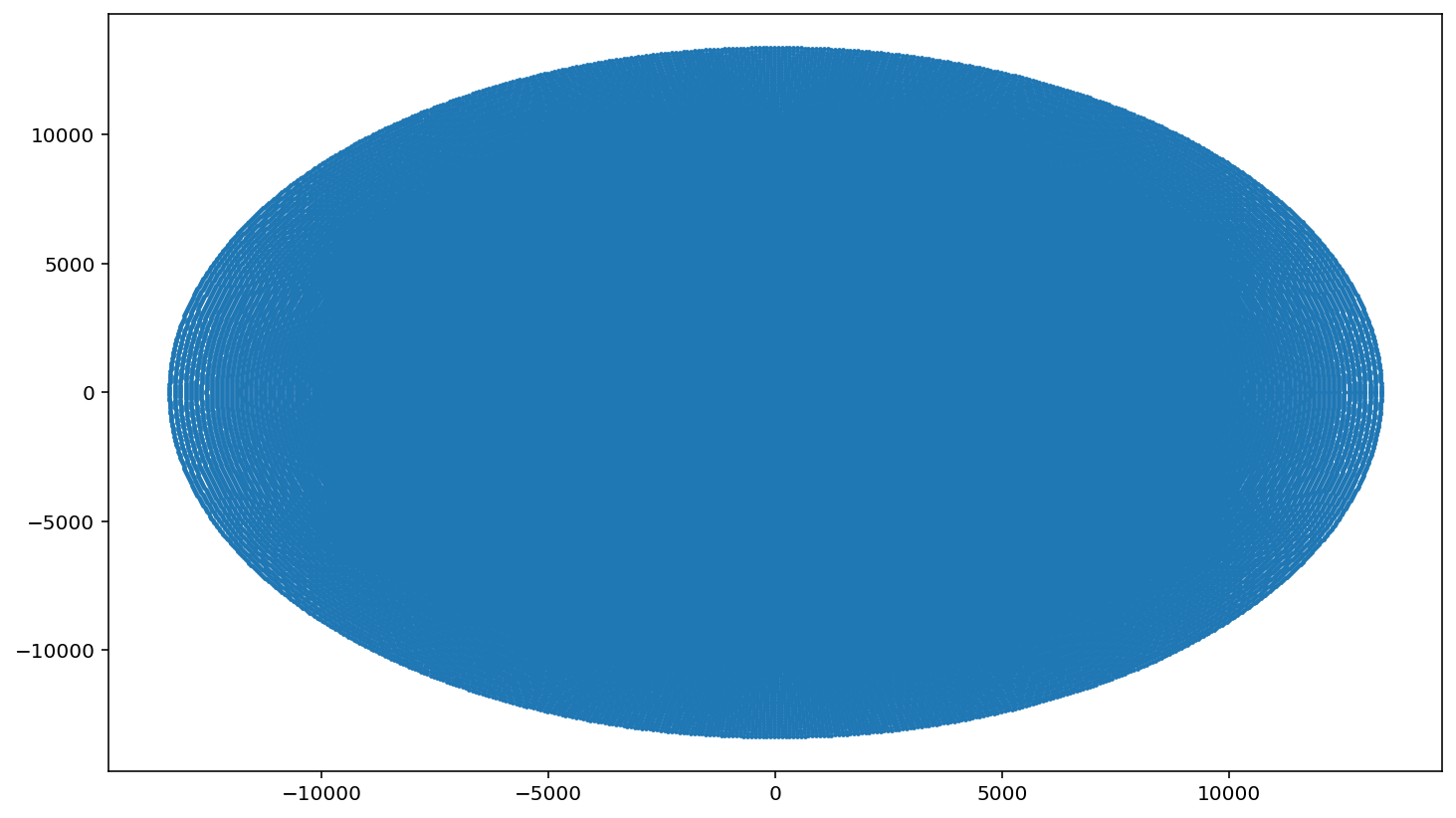}
  	\label{fig3}
  	 \end{figure}
 
\begin{theorem*}
 Let $\phi(z) = \zeta z + a$ with $|\zeta| < 1$ and $a\neq 0$, for which values of $\zeta$ and $a$, $\text{Ber}(C_\phi)$ is convex?
\end{theorem*}

\section{Composition Operator on Dirichlet space}	
Let $\mathbb{D}$ denote the open unit disc, and Hol($\mathbb{D}$) be the set of all holomorphic functions in $\mathbb{D}$. Let $\text{Aut}(\mathbb{D})$ be the set of all automorphisms on $\mathbb{D}$. Let $dA$ denote the area Lebesgue measure on the complex plane $\mathbb{C}$. Given a function $f\in$ Hol($\mathbb{D}$), the \textit{Dirichlet integral} of $f$ is defined by
$$ \mathcal{D}(f):= \frac{1}{\pi} \int_{\mathbb{D}} |f^{'}(z)|^{2} dA(z).$$
Then we define \textit{Dirichlet space} $\mathcal{D}$ \cite{el2014primer} to be the vector space of all functions $f \in$ Hol$(\mathbb{D}$) such that $ \mathcal{D}(f) < \infty.$ The Dirichlet space is a subset of the Hardy space, and $\mathcal{D}$ is dense in the Hardy space $H^2(\mathbb{D})$ as it contains the polynomials. The Dirichlet space $\mathcal{D}$ is a reproducing kernel Hilbert space with a norm induced by the inner product
$$\langle f,f \rangle_{\mathcal{D}}:= \langle f,f\rangle_{H^2(\mathbb{D})} +\mathcal{D}(f),$$
where $\langle .,.\rangle_{H^2(\mathbb{D})}$ is the usual inner product in ${H^2(\mathbb{D})}$ and the reproducing kernel is defined as
$$ k_w(z) = \frac{1}{\bar{w}z}\log\left(\frac{1}{1-\bar{w}z}\right) $$
for $w\neq 0$ and $k_0 \equiv 1$. Observe that for $z\neq 0$, we have
\begin{center}
  	\begin{equation*}
  		\begin{split}
  			k_z(z) &=\langle k_{z},k_{z}\rangle = ||k_z||^2\\
  			&=\frac{1}{|z|^2}\log\left(\frac{1}{1-|z|^2}\right).
  		\end{split}
  	\end{equation*}
  \end{center}
Let $\phi$ be a complex-valued function $\phi : \mathbb{D}\longrightarrow \mathbb{D}$. A composition operator $C_\phi$ acting on the Dirichlet space $\mathcal{D}$, will have the Berezin transform
\begin{center}
  	\begin{equation*}
  		\begin{split}
  			\widetilde{C_{\phi}}(z) &=\langle C_{\phi}\hat{k}_{z},\hat{k}_{z}\rangle\\
  			&=\frac{1}{||k_z||^2}\langle C_{\phi}{k_{z}},{k_{z}}\rangle\\
  			&=\frac{1}{||k_z||^2} k_{z}(\phi(z)).
  		\end{split}
  	\end{equation*}
  \end{center}
  
 \begin{theorem}[\cite{el2014primer}, Section 1.4]
 If $f \in \text{Hol}(\mathbb{D})$ and $\phi \in \text{Aut}(\mathbb{D})$, then $\mathcal{D}(f \circ \phi) = \mathcal{D}(f)$. Consequently, if $f \in \mathcal{D}$, then also $C_\phi f=f \circ \phi \in \mathcal{D}$.
 \end{theorem}
 Thus, if $\phi \in \text{Aut}(\mathbb{D})$, then the composition operator is bounded. Any automorphism
on the unit disc $\mathbb{D}$ is of the form
$$\phi (z) = e^{i\theta} \frac{\alpha-z}{1- \bar{\alpha}z}$$
where $\alpha \in \mathbb{D}$. Now we consider the automorphisms, elliptic symbol, and Blaschke factor on the unit disc $\mathbb{D}$ and characterize the convexity of the Berezin range of their corresponding composition operators on the Dirichlet space $\mathcal{D}$.

\subsection{Elliptic symbol}
For $z\in \mathbb{D}$ and $\zeta\in \mathbb{T}$, we define a self-map on the unit disc, $ \phi(z) = \zeta z$. Consider the composition operator $ C_{\phi}f$ acting on the reproducing kernel Hilbert space $\mathcal{D}$. The Berezin transform of $C_\phi$ can be computed as 
\begin{center}
  	\begin{equation*}
  		\begin{split}
  			\widetilde{C_{\phi}}(z) &=\langle C_{\phi}\hat{k}_{z},\hat{k}_{z}\rangle\\  			
  			&=\frac{1}{||k_z||^2} k_{z}(\phi(z)).
  		\end{split}
  	\end{equation*}
  \end{center}
  For $z\neq 0$, we have
  $$||k_z||^2 =\frac{1}{|z|^2} \log\left(\frac{1}{1-|z|^2}\right)$$
  		
 and
 $$k_{z}(\phi(z))=\frac{1}{\bar{z}\phi(z)} \log\left(\frac{1}{1-\bar{z}\phi(z)}\right).$$
 Therefore, when $\phi(z) = \zeta z$, $\zeta
 \in \mathbb{T}$ and $z \neq 0$, the Berezin transform of $C_{\phi}$ is
 \begin{center}
  	\begin{equation*}
  		\begin{split}
  			\widetilde{C_{\phi}}(z) &=\frac{1}{||k_z||^2} k_{z}(\phi(z))\\  			
  			&=\frac{1}{\frac{1}{|z|^2} \log\left(\frac{1}{1-|z|^2}\right)} \frac{1}{\zeta|z|^2} \log\left(\frac{1}{1-\zeta|z|^2}\right)\\
  			&=\frac{1}{\zeta}\frac{\log(1-\zeta|z|^2)}{\log(1-|z|^2)}\\
  			&=\frac{\log(1-\zeta|z|^2)}{\zeta \log(1-|z|^2)}
  		\end{split}
  	\end{equation*}
  \end{center}
  and for $z=0$, $\widetilde{C_{\phi}}(0) = 1.$

 The Berezin range of these composition operators may not always be convex, as shown in Figure \ref{fig4}. The following theorem gives the values of $\zeta$ for which $B(C_\phi)$ is convex.

  \begin{figure}[h]
    	\caption{$B(C_\phi)$ on $\mathcal{D}$ for $\zeta=-i$ (apparently not convex).}
  	\includegraphics[scale=0.5]{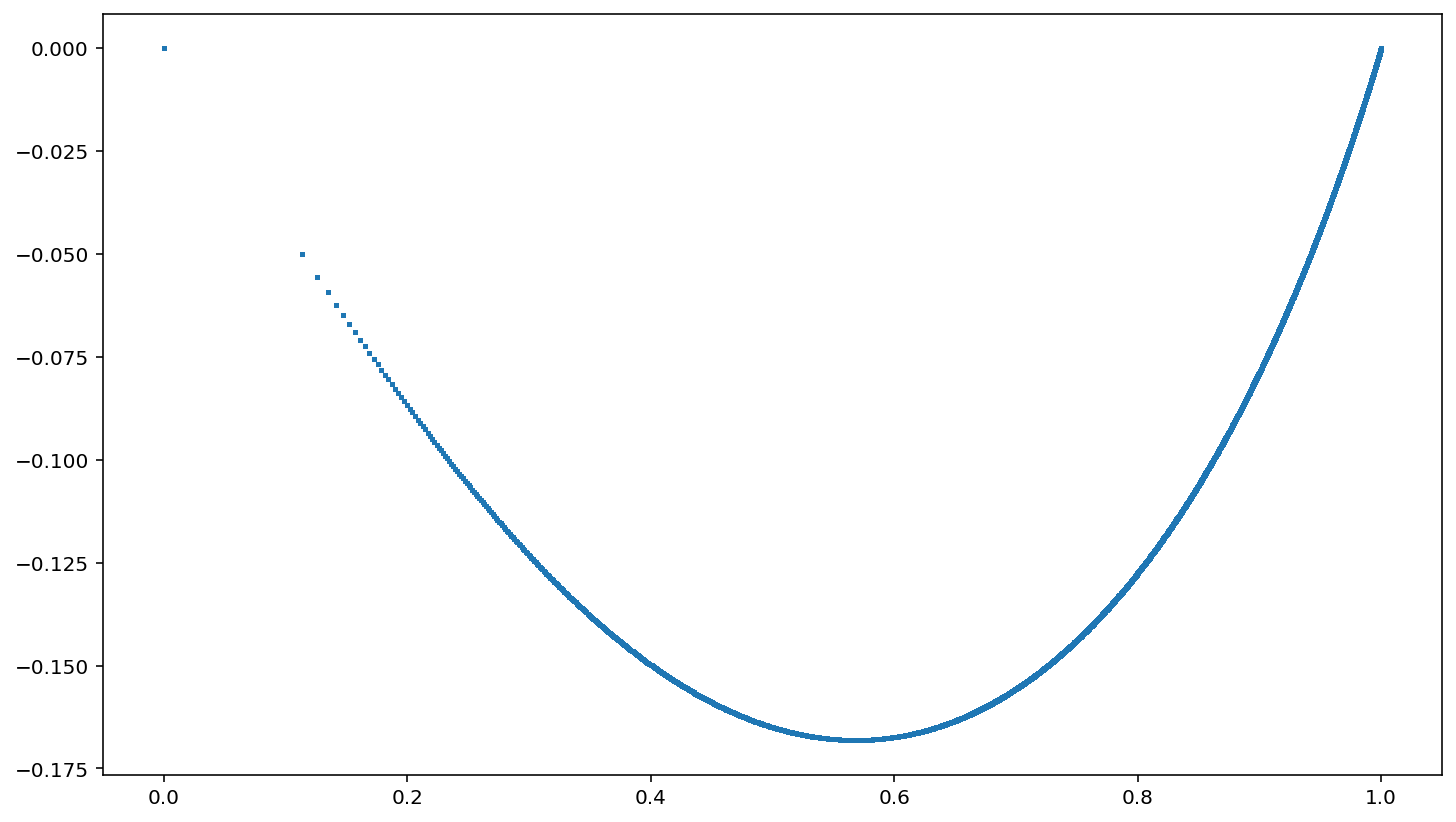}
  	\label{fig4}    
  	 \end{figure}

\begin{theorem}\label{ellipticdiri}                                                                                                                                                                                                                                                                                                                                                                                                                                                                                                                                                                                                                                                                                                                                                                                                                                                                                                                                                                                                                                                                                                                                                                                                                                                                                                                                                                                                                                                                                                                                                                                                                                                                                                                                                                                                                                                                                                                                                                                                                                                                                                                                                                                                                                                                                                                                                                                                                                                                                                                                                                     
Let $z\in \mathbb{D}$,  $\zeta \in \mathbb{T}$ and $\phi(z)=\zeta z$. Then the Berezin range of $C_\phi$ acting on the Dirichlet space $\mathcal{D}$ is convex if and only if $\zeta \in \{-1,1\}$.
\end{theorem}
\begin{proof}
First, let us prove the backward implication. For $\zeta =1$, we have $\widetilde{C_{\phi}}(0) = 1$ and 
 $$ \widetilde{C_{\phi}}(z) = \frac{\log(1-\zeta|z|^2)}{\zeta \log(1-|z|^2)} = 1$$
 for $z\neq 0$. So $\text{Ber}(C_\phi)= \{1\},$ which is a singleton set and is therefore convex. 
 
 Now for $\zeta = -1$, we have $\widetilde{C_{\phi}}(0) = 1$ and 
 $$ \widetilde{C_{\phi}}(z) = \frac{\log(1-\zeta|z|^2)}{\zeta \log(1-|z|^2)} = - \frac{\log(1+|z|^2)}{ \log(1-|z|^2)}.$$
 Put $z= re^{i\theta}$. Then for $0<r<1$, we have
 $$ \widetilde{C_{\phi}}(re^{i\theta}) = - \frac{\log(1+r^2)}{ \log(1-r  ^2)}.$$
 Therefore we obtain
 $$ \text{Ber}(C_\phi) = \left\lbrace - \frac{\log(1+r^2)}{ \log(1-r  ^2)} : r \in (0,1) \right\rbrace \cup \{1\} = (0,1],$$
 which is also a convex set in $\mathbb{C}$.
 
  Conversely, assume that the Berezin range of $C_\phi$ acting on the Dirichlet space $\mathcal{D}$ is convex.  Let $z\in \mathbb{D}$,  $\zeta \in \mathbb{T}$ and $\phi(z)=\zeta z$. For $z\neq 0$, we have the Berezin transform
 $$ \widetilde{C_{\phi}}(z) = \frac{\log(1-\zeta|z|^2)}{\zeta \log(1-|z|^2)}$$
and for $z=0$, $\widetilde{C_{\phi}}(0) = 1.$ Put $z= re^{i\theta}$.  Then we get  
 $$ \widetilde{C_{\phi}}(re^{i\theta}) = \frac{\log(1-\zeta r^2)}{\zeta \log(1-r^2)},$$
 which is a function that is independent of $\theta$. Therefore,  $\text{Ber}(C_\phi)$ is a just path in the complex plane. Since we have assumed convexity of  $\text{Ber}(C_\phi)$, a convex path in $\mathbb{C}$ must be either a point or a line segment. It is easy to observe that $\text{Ber}(C_\phi)$ is a point if and only if $\zeta = 1$. So assume that $\text{Ber}(C_\phi)$ is a line segment. Since $\widetilde{C_{\phi}}(0) = 1$ and $\lim_{r\rightarrow 1^-}\widetilde{C_{\phi}}(re^{i\theta}) = 0$, $\text{Ber}(C_\phi)$ is a line segment passing through the point 1 and approaching zero. Therefore, the line segment is contained in the real line, and the imaginary part of the Berezin transform is zero. The imaginary part of $\text{Ber}(C_\phi)$ is zero if and only if the imaginary part of $\zeta $ is zero. Thus, as $\zeta \in \mathbb{T}$, we have $\zeta = 1$ or $\zeta= -1.$
\end{proof}

\subsection{Blaschke Factor}
Consider the automorphism on the unit disc known as the Blaschke factor
$$\phi_\alpha(z) = \frac{z-\alpha}{1-\overline{\alpha}z}
$$
where $\alpha\in \mathbb{D}$	  and  consider the composition operator $C_{\phi_\alpha}f = f\circ \phi_\alpha.$ Berezin range of these composition operators may not always be convex, as shown in Figure \ref{fig5}. We try to find the values of $\alpha$ for which $\text{Ber}(C_{\phi_\alpha})$ is convex.

  \begin{figure}[h]
    	\caption{$\text{Ber}(C_{\phi_\alpha})$ on $\mathcal{D}$ for $\alpha=\frac{1}{2}e^{(i\frac{\pi}{3})}$ (apparently not convex).}
  	\includegraphics[scale=0.5]{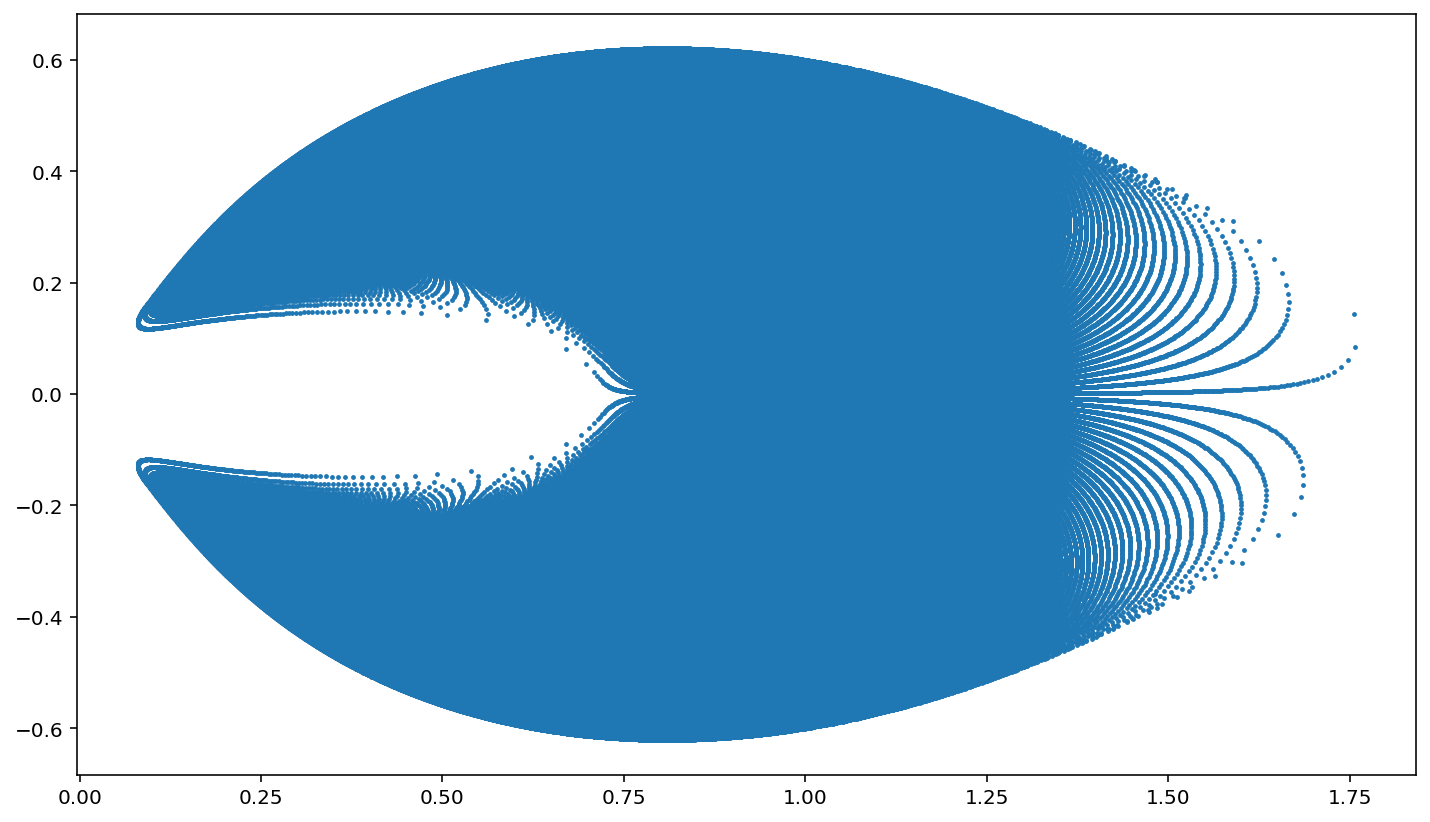}
  	\label{fig5}
  	 \end{figure}
  	 We have the Berezin transform
  	 \begin{center}
  	\begin{equation*}
  		\begin{split}
  			\widetilde{C}_{\phi_\alpha}(z) &=\frac{1}{||k_z||^2} k_{z}(\phi(z))\\
  			&=\frac{1}{\frac{1}{|z|^2}\log\left(\frac{1}{1-|z|^2}\right)}\frac{1}{\bar{z}\frac{z-\alpha}{1-\overline{\alpha}z}}\log\left(\frac{1}{1 - \bar{z}\frac{z-\alpha}{1-\overline{\alpha}z}}\right).
  		\end{split}
  	\end{equation*}
  \end{center}
We now compute each separately. Since $Re\{\bar{z}\alpha\} = Re\{\bar{\alpha}z\}$ and $Im\{\bar{z}\alpha\} = -Im\{\bar{\alpha}z\}$, we have  
	\begin{center}
  	\begin{equation*}
  		\begin{split}
  			\frac{1}{\bar{z}\frac{z-\alpha}{1-\overline{\alpha}z}} &=\frac{1-\bar{\alpha}z}{|z|^2-\bar{z}\alpha}\\
  			&=\frac{1-\bar{\alpha}z}{|z|^2-Re\{\bar{z}\alpha\} - iIm\{\bar{z}\alpha\}}\\
  			&=\frac{(1-\bar{\alpha}z)(|z|^2-Re\{\bar{z}\alpha\} + iIm\{\bar{z}\alpha\})}{|~~|z|^2-Re\{\bar{z}\alpha\} - iIm\{\bar{z}\alpha\}|^2}\\
  			&=a_{\alpha,z}(|z|^2-Re\{\bar{z}\alpha\} + iIm\{\bar{z}\alpha\} - (Re\{\bar{\alpha}z\} + iIm\{\bar{\alpha}z\})(|z|^2-Re\{\bar{z}\alpha\})\\
  			&\qquad -i(Re\{\bar{\alpha}z\} + iIm\{\bar{\alpha}z\})Im\{\bar{z}\alpha\})\\
  			&=a_{\alpha,z}((|z|^2-Re\{\bar{\alpha}z\})(1 - Re\{\bar{\alpha}z\})- (Im\{\bar{\alpha}z\})^2\\
  			&\qquad +iIm\{\bar{\alpha}z\}(2Re\{\bar{\alpha}z\}- |z|^2-1))
  		\end{split}
  	\end{equation*}
  \end{center}
  where
  $$a_{\alpha,z}= \frac{1}{|~~|z|^2-Re\{\bar{z}\alpha\} - iIm\{\bar{z}\alpha\}|^2}.$$
Similarly
\begin{center}
  	\begin{equation*}
  		\begin{split}
 			 \frac{1}{1-\bar{z}\frac{z-\alpha}{1-\overline{\alpha}z}} &= \frac{(1 - \alpha \overline{z})}{1 - \overline{\alpha} z -|z|^2 + \overline{z}\alpha}\\
 			 &= \frac{(1 - \overline{\alpha} z)}{1-|z|^2 + 2iIm\{\overline{z}\alpha\}}\\
 			 &=b_{\alpha, z}(1 - \alpha \overline{z})(1 - |z|^2 - 2i \mathrm{Im}\{\bar{z}\alpha\})\\
 			 &=b_{\alpha, z}(1 - |z|^2 + 2i \mathrm{Im}\{\overline{\alpha} z\} - \overline{\alpha} z(1 - |z|^2) + 2i \mathrm{Im}\{\alpha \overline{z}\}\overline{\alpha} z)\\
 			 &= b_{\alpha, z}[1 - |z|^2 + 2i \mathrm{Im}\{\overline{\alpha} z\} - (Re\{\overline{\alpha} z\} + i \mathrm{Im}\{\overline{\alpha} z\})(1 - |z|^2)\\
 			 &\qquad - 2i \mathrm{Im}\{\alpha \overline{z}\}(Re\{\overline{\alpha} z\} + i \mathrm{Im}\{\overline{\alpha} z\})]\\
 			 &=b_{\alpha, z}[1 - |z|^2 + 2i \mathrm{Im}\{\overline{\alpha} z\} - (1 - |z|^2)Re\{\overline{\alpha} z\} - i(1 - |z|^2)\mathrm{Im}\{\overline{\alpha} z\} \\
 			 &\qquad- 2i \mathrm{Im}\{\alpha \overline{z}\}Re\{\overline{\alpha} z\} + 2(\mathrm{Im}\{\overline{\alpha} z\})^2]\\
 			 &= b_{\alpha, z}[1 - |z|^2 - (1 - |z|^2)Re\{\overline{\alpha} z\} + 2(\mathrm{Im}\{\overline{\alpha} z\})^2] \\
 			 &\qquad  + ib_{\alpha, z}[2\mathrm{Im}\{\overline{\alpha} z\} - (1 - |z|^2)\mathrm{Im}\{\overline{\alpha} z\} - 2\mathrm{Im}\{\alpha \overline{z}\}Re\{\overline{\alpha} z\}]\\
 			 &= b_{\alpha, z}[(1 - |z|^2)(1 - Re\{\overline{\alpha} z\}) + 2(\mathrm{Im}\{\overline{\alpha} z\})^2] \\
 			 &\qquad + ib_{\alpha, z}\mathrm{Im}\{\alpha \overline{z}\}[1 + |z|^2 - 2Re\{\overline{\alpha} z\}]
  		\end{split}
  	\end{equation*}
  \end{center}
where
$$ b_{\alpha, z} = \frac{1}{\left|1 - |z|^2 + 2i\mathrm{Im}\{\bar{z}\alpha\}\right|^2}.$$ 
So define
$$c_{\alpha, z} = \frac{a_{\alpha, z} }{\frac{1}{|z|^2} \log\left(\frac{1}{1-|z|^2}\right)}.$$
Then, the Berezin transform
\begin{center}
  	\begin{equation}\label{3.1}
  		\begin{split}
  			\widetilde{C}_{\phi_\alpha}(z) &= c_{\alpha, z}[(|z|^2 - \mathrm{Re}\{\overline{\alpha} z\})(1 - \mathrm{Re}\{\overline{\alpha} z\}) - (\mathrm{Im}\{\overline{\alpha} z\})^2\\
  			&\qquad + i\mathrm{Im}\{\overline{\alpha} z\}(2\mathrm{Re}\{\overline{\alpha} z\} - |z|^2 - 1)] \times\\
  			&\qquad \log[b_{\alpha, z}[(1 - |z|^2)(1 - \mathrm{Re}\{\overline{\alpha} z\}) + 2(\mathrm{Im}\{\overline{\alpha} z\})^2\\
  			&\qquad + i\mathrm{Im}\{\alpha \overline{z}\}(1 + |z|^2 - 2\mathrm{Re}\{\alpha \overline{z}\})]].
  		\end{split}
  	\end{equation}
  \end{center}
\begin{proposition}\label{conj}
  	The Berezin range of $C_{\phi_\alpha}$ acting on $\mathcal{D}$ is closed under complex conjugation and therefore is symmetric about the real axis.
  	\end{proposition}
 \begin{proof}
 We show that the Berezin transform for every $z \in \mathbb{D}$ is the conjugate of the Berezin transform for some other point in the open unit disc. Put $\alpha = \rho e^{i\psi}$ and $z=re^{i\theta}$. Then, we claim that $\widetilde{C}_{\phi_\alpha}(re^{i\theta})=\overline{\widetilde{C}_{\phi_\alpha}(re^{i(2\psi-\theta)})}$. Note that 
   \begin{center}
  	\begin{equation*}
  		\begin{split}
  			||k_z||^2 &= \langle k_z,k_z\rangle\\  			
  			&=k_z(z)\\
  			&= \frac{1}{|z|^2} \log\left(\frac{1}{1-|z|^2}\right)
  		\end{split}
  	\end{equation*}
  \end{center}
  and
 $$ k_{z}(\phi_\alpha(z))=\frac{1}{\bar{z}\phi_\alpha(z)} \log\left(\frac{1}{1-\bar{z}\phi_\alpha(z)}\right)$$
  where $k_{w}(z)=\frac{1}{\bar{w}z} \log\left(\frac{1}{1-\bar{w}z}\right).$ Then we have
 \begin{center}
  	\begin{equation*}
  		\begin{split}
  			\widetilde{C}_{\phi_\alpha}(z) &=\langle C_{\phi_\alpha}\hat{k}_{z},\hat{k}_{z}\rangle\\  			
  			&=\frac{1}{||k_z||^2} k_{z}(\phi_\alpha(z))\\
  			&=\frac{1}{\frac{1}{|z|^2} \log\left(\frac{1}{1-|z|^2}\right)}\frac{1}{\bar{z}\phi_\alpha(z)} \log\left(\frac{1}{1-\bar{z}\phi_\alpha(z)}\right)
  		\end{split}
  	\end{equation*}
  \end{center}
 
 Put $z=re^{i\theta}$. Then we compute $\widetilde{C}_{\phi_\alpha}(re^{i\theta})$.
 $$\widetilde{C}_{\phi_\alpha}(re^{i\theta})   			
  			= \frac{1}{\frac{1}{r^2} \log\left(\frac{1}{1-r^2}\right)}\frac{1}{re^{-i\theta}\phi_\alpha(re^{i\theta})} \log\left(\frac{1}{1-re^{-i\theta}\phi_\alpha(re^{i\theta})}\right).$$
  			
 Similarly, put $z=re^{i(2\psi-\theta)}$. Then we compute $\overline{\widetilde{C_{\phi_\alpha}}(re^{i(2\psi-\theta)})}$.
 $$\overline{\widetilde{C_{\phi_\alpha}}(re^{i(2\psi-\theta)})} 			
  			= \frac{1}{\frac{1}{r^2} \log\left(\frac{1}{1-r^2}\right)}\frac{1}{re^{i(2\psi-\theta)}\overline{\phi_\alpha(re^{i(2\psi-\theta)})}} \log\left(\frac{1}{1-re^{i(2\psi-\theta)}\overline{\phi_\alpha(re^{i(2\psi-\theta)})}}\right).$$
  			
Also note that $\log(\bar{z}) = \overline{\log(z)}$. So our claim $\widetilde{C_{\phi_\alpha}}(re^{i\theta})=\overline{\widetilde{C_{\phi_\alpha}}(re^{i(2\psi-\theta)})}$ is true if and only if
$$\frac{1}{\phi_\alpha(re^{i\theta})} \log\left(\frac{1}{1-re^{-i\theta}\phi_\alpha(re^{i\theta})}\right) = \frac{1}{e^{i2\psi}\overline{\phi_\alpha(re^{i(2\psi-\theta)})}} \log\left(\frac{1}{1-re^{i(2\psi-\theta)}\overline{\phi_\alpha(re^{i(2\psi-\theta)})}}\right).$$
So, we compute $e^{i2\psi}\overline{\phi_\alpha(re^{i(2\psi-\theta)})}$:

\begin{center}
		\begin{equation*}
			\begin{split}
				e^{i(2\psi)}\overline{\phi_\alpha(re^{i(2\psi-\theta)})} &= e^{i(2\psi)}\frac{re^{i(\theta-2\psi)}-\rho e^{-i\psi}}{1-\rho e^{i\psi}re^{i(\theta-2\psi)}}\\
				&=\frac{re^{i\theta}-\rho e^{i\psi}}{1-\rho e^{-i\psi}re^{i\theta}}\\
				&=\phi_\alpha(re^{i\theta}).
			\end{split}
		\end{equation*}
	\end{center}
	Therefore, we get $$\frac{1}{\phi_\alpha(re^{i\theta})}= \frac{1}{e^{i(2\psi)}\overline{\phi_\alpha(re^{i(2\psi-\theta)})}}.$$
Since $re^{-i\theta}\phi_\alpha(re^{i\theta}) = re^{i(2\psi-\theta)} \overline{\phi_\alpha(re^{i(2\psi-\theta)})}$, we have
	$$ \log\left(\frac{1}{1-re^{-i\theta}\phi_\alpha(re^{i\theta})}\right) = \log\left(\frac{1}{1-re^{i(2\psi-\theta)}\overline{\phi_\alpha(re^{i(2\psi-\theta)})}}\right).$$
	Therefore $\widetilde{C_{\phi_\alpha}}(re^{i\theta})=\overline{\widetilde{C_{\phi_\alpha}}(re^{i(2\psi-\theta)})}$.
 \end{proof}
 Now, we prove an immediate consequence of the above proposition that will be important in establishing our main result of this section.
 \begin{corollary}\label{real}
	If the Berezin range of $C_{\phi_\alpha}$ acting on $\mathcal{D}$ is convex, then $Re\{\widetilde{C}_{\phi_\alpha}(z)\} \in \text{Ber}(C_{\phi_\alpha})$ for each $z\in \mathbb{D}$.
\end{corollary}

\begin{proof}
	Suppose $\text{Ber}(C_{\phi_\alpha})$ is convex. Then from Proposition \ref{conj} $\text{Ber}(C_{\phi_\alpha})$ is closed under complex conjugation. Therefore we have
	\begin{center}
		$\frac{1}{2}\widetilde{C}_{\phi_\alpha}(z) + \frac{1}{2}\overline{\widetilde{C}_{\phi_\alpha}(z)} = Re\{\widetilde{C}_{\phi_\alpha}(z)\} \in \text{Ber}(C_{\phi_\alpha})$.
	\end{center}
\end{proof}
Now, we use these tools proved above to characterize the convexity of the Berezin range for the Blaschke factor.
\begin{theorem}
Let $\mathrm{Im}\{\widetilde{C}_{\phi_\alpha}(z)\} = 0$  only when $\mathrm{Im}\{\overline{\alpha} z\} = 0$.  Then the Berezin range of $C_{\phi_\alpha}$ acting on  $\mathcal{D}$ is convex if and only if $\alpha = 0$.
\end{theorem}
\begin{proof}
If $\alpha = 0$, then $\widetilde{C}_{\phi_\alpha}(z) = 1$ for all $z \in \mathbb{D}$. So $\text{Ber}(C_{\phi_\alpha}) = \{1\}$, which is a singleton set in  $\mathbb{C}$ and is therefore convex. Conversely, suppose that $\text{Ber}(C_{\phi_\alpha})$ is convex. So by Corollary \ref{real} for each $z \in \mathbb{D}$, we can find  $w \in \mathbb{D}$ such that
$$\widetilde{C}_{\phi_\alpha}(w) = Re\{\widetilde{C}_{\phi_\alpha}(z)\}.$$
This gives $Im\{\widetilde{C}_{\phi_\alpha}(w)\} = 0.$ This happens if and only if $Im\{\overline{\alpha} w\} = 0$. This implies that $w \text{ and } \alpha$ lie on a line passing through the origin. So put $w = r\alpha, \text{ where } r \in (-1/|\alpha|, 1/|\alpha|)$. Then we have
\begin{center}
		\begin{equation*}
			\begin{split}
				\widetilde{C}_{\phi_\alpha}(w) &= Re\{\widetilde{C}_{\phi_\alpha}(z)\}\\
				&=c_{\alpha, r\alpha}[(|r\alpha|^2 - Re\{ \overline{\alpha} r\alpha \})(1 - Re\{\overline{\alpha} r\alpha\}) - (Im\{\overline{\alpha} r\alpha\})^2\\
				&\qquad + iIm\{\overline{\alpha} r\alpha\}(2Re\{\overline{\alpha} r\alpha\} - |r\alpha|^2 - 1)] \times\\
				&\qquad \log[b_{\alpha, z}[(1 - |r\alpha|^2)(1 - Re\{\overline{\alpha} r\alpha\}) + 2(Im\{\overline{\alpha} r\alpha\})^2]\\
				&\qquad + iIm\{\alpha \overline{r\alpha}\}[1 + |r\alpha|^2 - 2Re\{\overline{\alpha} r\alpha\}]]\\
				&= \frac{|r\alpha|^2 (|r\alpha|^2 - r| \alpha|^2)(1 - r|\alpha|^2) \log\left[\frac{(1 - |r\alpha|^2)(1 - r|\alpha|^2)}{(1 - |r\alpha|^2)^2}\right]}{(|r\alpha|^2 - r|\alpha|^2)^2 \log\left(\frac{1}{{1 - |r\alpha|^2}}\right)}\\
				&= \frac{|r\alpha|^2(1 - r|\alpha|^2)[\log(1 - |r\alpha|^2) - \log(1 - r|\alpha|^2)]}{{(|r\alpha|^2 - r|\alpha|^2) \log(1 - |r\alpha|^2)}}\\
				&= \frac{|r\alpha|^2(1 - r|\alpha|^2)}{(|r\alpha|^2 - r|\alpha|^2)}\left(1 - \frac{\log(1 - r|\alpha|^2)}{\log(1 - |r\alpha|^2)}\right).
			\end{split}
		\end{equation*}
	\end{center}
Observe that for $\alpha \neq 0,$ 
$$ \lim_{r \rightarrow \frac{1}{|\alpha|}} \widetilde{C}_{\phi_\alpha}(w) = 1.$$
Since $w=r\alpha$, as $r \rightarrow \frac{1}{|\alpha|}$, $|w| \rightarrow 1$. Therefore $\widetilde{C}_{\phi_\alpha}(w) = Re\{\widetilde{C}_{\phi_\alpha}(z)\} \rightarrow 1$ as $|w| \rightarrow 1$. But from the expression of $\widetilde{C}_{\phi_\alpha}(z)$ in Equation \ref{3.1}, putting $z=\rho e^{i\theta}$, we get
$$ \lim_{\rho \rightarrow 1^-} \widetilde{C}_{\phi_\alpha}(\rho e^{i\theta}) = 0.$$
This is a contradiction as the $Re\{\widetilde{C}_{\phi_\alpha}(\rho e^{i\theta})\} \rightarrow 0 \neq 1$ as $|\rho e^{i\theta}|=\rho \rightarrow 1$. Hence $\text{Ber}(C_{\phi_\alpha})$ cannot be convex unless $\alpha = 0$.
\end{proof}
\section{Unitary Equivalent Berezin range}
Let $\mathcal{H}$ be a reproducing kernel Hilbert space.  Then, for a bounded linear operator $T$ on $\mathcal{H}$, \textit{unitarily equivalent Berezin range} $\mathcal{B}(T)$ \cite{nordgren1994boundary} is defined as the set of all Berezin transforms of all operators that are unitarily equivalent to $T$. Note that for any two unitarily equivalent operators $A$ and $T$, we have $\mathcal{B}(A) =\mathcal{B}(T)$. 

\begin{theorem} \cite{gustafson1997numerical} (Elliptic range theorem)
Let $A$ be a $2\times 2$ matrix with complex entries and eigenvalues $\lambda_1$ and $\lambda_2$. Then the numerical range $W(A)$ is an elliptic disc with foci $\lambda_1$ and $\lambda_2$. 

\end{theorem}
\begin{remark}\label{ellipse}
Without loss of generality, consider an upper triangular matrix
$$
	A=\begin{bmatrix} 
	\lambda_1 & m \\
	0 & \lambda_2\\
	\end{bmatrix},
	$$
	where $\lambda_1$ and $\lambda_2$ are the eigen values of $A$. Here we discuss the shape of the $W(A)$ for different values of $\lambda_1,$ $\lambda_2$ and $m$.   If $\lambda_1 \neq \lambda_2$ and $m\neq 0$, the numerical range of $T$, $W(T)$ is an ellipse with foci at $\lambda_1,\lambda_2$, minor axis $|m|$ and major axis $\sqrt{4r^2+|m|^2}$, where $\frac{\lambda_1 - \lambda_2}{2} = re^{i\mu}$. If $\lambda_1 = \lambda_2 = \lambda$, then $W(A)$ is a circle with centre at $\lambda$ and radius $\frac{|m|}{2}$. If $\lambda_1 \neq \lambda_2$ and $m=0$, then $W(A)$ is the set of all convex combinations of $\lambda_1$ and $\lambda_2$ and is the line segment joining them. 
	\end{remark}
	\begin{lemma}\label{sub}
Let $T$ be a bounded linear operator on a reproducing kernel Hilbert space $\mathcal{H}$. Then	$\mathcal{B}(T) \subseteq W(T)$.
	\end{lemma}
	\begin{proof}
	Let $\alpha \in \mathcal{B}(T).$ Then  there exist a normalized reproducing kernel $\hat{k}_\alpha \in \mathcal{H}$ such that
$$\alpha = \langle A \hat{k}_\alpha,\hat{k}_\alpha\rangle_{\mathcal{H}}$$
where $A$ is an operators on $\mathcal{H}$ that is unitarily equivalent to $T$. Since $||\hat{k}_\alpha||=1$, it is easy to observe that $\alpha \in W(A)$. Also, note that $A$ is unitarily equivalent to $T$. Therefore $W(A)=W(T)$. This implies that $\alpha \in W(T)$. Hence $\mathcal{B}(T) \subseteq W(T)$.
	\end{proof}
	Now, we prove an analogous result of the elliptic range theorem for the unitarily equivalent Berezin range. This proof explicitly provides the $2\times 2$ unitary matrix corresponding to each element in the numerical range.

	\begin{theorem}\label{range}
	Let $T$ be a $2\times 2$ matrix with complex entries and eigenvalues $\lambda_1$ and $\lambda_2$. Then, the unitarily equivalent Berezin range $\mathcal{B}(T)$  is an elliptic disc with foci $\lambda_1$ and $\lambda_2$.
	\end{theorem}
	\begin{proof}
	
Without loss of generality, we can consider $T$ as an upper triangular matrix
$$
	\begin{bmatrix} 
	\lambda_1 & m \\
	0 & \lambda_2\\
	\end{bmatrix},
	$$
	where $\lambda_1$ and $\lambda_2$ are the eigen values of $T$. By Lemma \ref{sub}, $\mathcal{B}(T) \subseteq W(T)$. Here we prove that $W(T) \subseteq \mathcal{B}(T)$ and the  result follows.
	
If $\lambda_1\neq \lambda_2$ and $m=0$ we have 
	$$T= 
	\begin{bmatrix} 
	\lambda_1 & 0 \\
	0 & \lambda_2\\
	\end{bmatrix},
	$$
	and by the Remark \ref{ellipse}, $W(T)$ is the line segment joining $\lambda_1$ and $\lambda_2$. So every element in $W(T)$ will be of the form  $k\lambda_1 + (1-k) \lambda_2$, where $k\in [0,1]$. Define the matrix
	$$U= 
	\begin{bmatrix} 
	\sqrt{k} & \sqrt{1-k} \\
	\sqrt{1-k} & -\sqrt{k}\\
	\end{bmatrix},
	$$
where $k\in[0,1]$. Observe that 
$$U^*U=UU^*= 
	\begin{bmatrix} 
	\sqrt{k} & \sqrt{1-k} \\
	\sqrt{1-k} & -\sqrt{k}\\
	\end{bmatrix}\begin{bmatrix} 
	\sqrt{k} & \sqrt{1-k} \\
	\sqrt{1-k} & -\sqrt{k}\\
	\end{bmatrix}=\begin{bmatrix} 
	1 & 0 \\
	0 & 1\\
	\end{bmatrix}.
	$$
Thus, the matrix $U$ is unitary. We compute $U^*TU$.
	\begin{equation*}
	\begin{split}
	U^*TU &=\begin{bmatrix} 
	\sqrt{k} & \sqrt{1-k} \\
	\sqrt{1-k} & -\sqrt{k}\\
	\end{bmatrix}
	\begin{bmatrix} 
	\lambda_1 & 0 \\
	0 & \lambda_2\\
	\end{bmatrix}\begin{bmatrix} 
	\sqrt{k} & \sqrt{1-k} \\
	\sqrt{1-k} & -\sqrt{k}\\
	\end{bmatrix}\\
	&=\begin{bmatrix} 
	k\lambda_1 + (1-k)\lambda_2 & \sqrt{k} \sqrt{1-k}(\lambda_1-\lambda_2) \\
	\sqrt{k} \sqrt{1-k}(\lambda_1-\lambda_2) & (1-k)\lambda_1 + k\lambda_2\\
	\end{bmatrix}.
	\end{split}
	\end{equation*}
Since the Berezin range of a finite-dimensional matrix is its diagonal elements \cite[Section 3.1]{cowen22}, the diagonal elements of  $U^*TU$ belong to $\mathcal{B}(T).$ So $k\lambda_1 + (1-k)\lambda_2 \in \mathcal{B}(T) $. This implies that $W(T)\subseteq \mathcal{B}(T)$. Therefore $\mathcal{B}(T)$ is the line segment joining $\lambda_1$ and $\lambda_2$. 

Now consider the case $\lambda_1=\lambda_2=\lambda$ and $m=|m|e^{i\delta} \neq 0$. Then we have
$$T= 
	\begin{bmatrix} 
	\lambda & m \\
	0 & \lambda\\
	\end{bmatrix},
	$$
	and by the Remark \ref{ellipse}, $W(T)$ is a disc with centre at $\lambda$ and radius $\frac{|m|}{2}$. Therefore every element in $W(T)$ is of the form $\lambda+re^{i\theta} $ where $r\in [0,\frac{|m|}{2}]$. Define the matrix
	$$U= 
	\begin{bmatrix} 
	e^{-i(\theta-\delta)} \sin\alpha & \cos\alpha \\
	\cos\alpha & -e^{i(\theta-\delta)} \sin\alpha\\
	\end{bmatrix},
	$$
where $\alpha=\frac{1}{2}\sin^{-1}\left(\frac{2r}{|m|}\right)$. Observe that 
$$U^*U=
	\begin{bmatrix} 
	e^{i(\theta-\delta)} \sin\alpha & \cos\alpha \\
	\cos\alpha & -e^{-i(\theta-\delta)} \sin\alpha\\
	\end{bmatrix}\begin{bmatrix} 
	e^{-i(\theta-\delta)} \sin\alpha & \cos\alpha \\
	\cos\alpha & -e^{i(\theta-\delta)} \sin\alpha\\
	\end{bmatrix}=\begin{bmatrix} 
	1 & 0 \\
	0 & 1\\
	\end{bmatrix}
	$$
and
$$UU^*=
	\begin{bmatrix} 
	e^{-i(\theta-\delta)} \sin\alpha & \cos\alpha \\
	\cos\alpha & -e^{i(\theta-\delta)} \sin\alpha\\
	\end{bmatrix}\begin{bmatrix} 
	e^{i(\theta-\delta)} \sin\alpha & \cos\alpha \\
	\cos\alpha & -e^{-i(\theta-\delta)} \sin\alpha\\
	\end{bmatrix}=\begin{bmatrix} 
	1 & 0 \\
	0 & 1\\
	\end{bmatrix}.
	$$
So, the matrix $U$ is unitary. We compute $U^*TU$. 
	\begin{equation*}
	\begin{split}
	U^*TU &=\begin{bmatrix} 
	e^{i(\theta-\delta)} \sin\alpha & \cos\alpha \\
	\cos\alpha & -e^{-i(\theta-\delta)} \sin\alpha\\
	\end{bmatrix}
	\begin{bmatrix} 
	\lambda & m \\
	0 & \lambda\\
	\end{bmatrix}\begin{bmatrix} 
	e^{-i(\theta-\delta)} \sin\alpha & \cos\alpha \\
	\cos\alpha & -e^{i(\theta-\delta)} \sin\alpha\\
	\end{bmatrix}\\
	&=\begin{bmatrix} 
	\lambda(\sin^2\alpha+\cos^2\alpha) + me^{i(\theta-\delta)}\sin\alpha \cos\alpha & -me^{i2\theta}\sin^2\alpha \\
	m\cos^2\alpha & \lambda(\sin^2\alpha+\cos^2\alpha) - me^{i(\theta-\delta)}\sin\alpha \cos\alpha\\
	\end{bmatrix}\\
	&= \begin{bmatrix} 
	\lambda + me^{i(\theta-\delta)}\sin\alpha \cos\alpha & -me^{i2\theta}\sin^2\alpha \\
	m\cos^2\alpha & \lambda - me^{i(\theta-\delta)}\sin\alpha \cos\alpha\\
	\end{bmatrix}.
	\end{split}
	\end{equation*}
Since the diagonal elements of  $U^*TU$ belongs to $\mathcal{B}(T)$, $\lambda + me^{i(\theta-\delta)}\sin\alpha \cos\alpha \in \mathcal{B}(T).$ Substituting $m=|m|e^{i\delta}$ and $\alpha=\frac{1}{2}\sin^{-1}\left(\frac{2r}{|m|}\right)$, we get

\begin{equation*}
	\begin{split}
\lambda + me^{i(\theta-\delta)}\sin\alpha \cos\alpha &= \lambda + |m|e^{i\delta}e^{i(\theta-\delta)}\frac{\sin(2\alpha)}{2} \\
&= \lambda + |m|e^{i\theta}\frac{\sin(\sin^{-1}\left(\frac{2r}{|m|}\right))}{2}\\
&= \lambda + re^{i\theta}.
\end{split}
	\end{equation*}
Thus $\lambda + re^{i\theta} \in \mathcal{B}(T).$  This implies that $W(T)\subseteq \mathcal{B}(T)$. Therefore, $\mathcal{B}(T)$ is the disc with centre at $\lambda$ and radius $\frac{|m|}{2}$.

Now if $\lambda_1 \neq \lambda_2$ and $m\neq 0$, we have 
$$T= 
	\begin{bmatrix} 
	\lambda_1 & m \\
	0 & \lambda_2\\
	\end{bmatrix}.
	$$
	Then by the Remark \ref{ellipse}, $W(T)$ is an ellipse with foci at $\lambda_1,\lambda_2$, minor axis $|m|$ and major axis $\sqrt{4r^2+|m|^2}$, where $\frac{\lambda_1 - \lambda_2}{2} = re^{i\mu}$. Consider
\begin{equation*}
	\begin{split}
	T - \frac{\lambda_1 + \lambda_2}{2}=& 
	\begin{bmatrix} 
	\frac{\lambda_1 - \lambda_2}{2} & m \\
	0 & \frac{\lambda_2 - \lambda_1}{2}\\
	\end{bmatrix},\\
	e^{-i\mu} \left(T - \frac{\lambda_1 + \lambda_2}{2}\right)=& 
	\begin{bmatrix} 
	r & me^{-i\mu} \\
	0 & -r\\
	\end{bmatrix} =A,
	\end{split}
	\end{equation*}
 Then by elliptic range theorem the numerical range of $A$, $W(A)$ is an ellipse with centre at $(0,0)$ and minor axis $|m|$, major axis $\sqrt{4r^2+|m|^2}$  and foci at $(r,0)$ and $(-r,0)$. We claim that $\mathcal{B}(A) = W(A)$ and therefore the result follows. 

Define a matrix
$$U = \begin{bmatrix} 
	e^{i\alpha}\cos\theta & e^{i\beta}\sin\theta \\
	e^{i\gamma}\sin\theta & e^{i\delta}\cos\theta\\
	\end{bmatrix},$$
where $\delta-\gamma = \pi +(\beta - \alpha)$. Since $ e^{i\zeta} + e^{i(\pi + \zeta)} = 0$, 
 $$U^*U = \begin{bmatrix} 
	e^{-i\alpha}\cos\theta & e^{-i\gamma}\sin\theta \\
	e^{-i\beta}\sin\theta & e^{-i\delta}\cos\theta\\
	\end{bmatrix}\begin{bmatrix} 
	e^{i\alpha}\cos\theta & e^{i\beta}\sin\theta \\
	e^{i\gamma}\sin\theta & e^{i\delta}\cos\theta\\
	\end{bmatrix}= \begin{bmatrix} 
	1 & 0 \\
	0 & 1\\
	\end{bmatrix},$$
and
$$UU^* = \begin{bmatrix} 
	e^{i\alpha}\cos\theta & e^{i\beta}\sin\theta \\
	e^{i\gamma}\sin\theta & e^{i\delta}\cos\theta\\
	\end{bmatrix}\begin{bmatrix} 
	e^{-i\alpha}\cos\theta & e^{-i\gamma}\sin\theta \\
	e^{-i\beta}\sin\theta & e^{-i\delta}\cos\theta\\
	\end{bmatrix}= \begin{bmatrix} 
	1 & 0 \\
	0 & 1\\
	\end{bmatrix}.$$
Therefore, $U$ is a unitary matrix. Now, for the unitary matrix $U$, we compute $U^*AU$.
$$U^*AU = \begin{bmatrix} 
	e^{-i\alpha}\cos\theta & e^{-i\gamma}\sin\theta \\
	e^{-i\beta}\sin\theta & e^{-i\delta}\cos\theta\\
	\end{bmatrix} \begin{bmatrix} 
	r & me^{-i\mu} \\
	0 & -r\\
	\end{bmatrix}\begin{bmatrix} 
	e^{i\alpha}\cos\theta & e^{i\beta}\sin\theta \\
	e^{i\gamma}\sin\theta & e^{i\delta}\cos\theta\\
	\end{bmatrix}$$
	$$\tiny{=\begin{bmatrix} 
	r \cos^2\theta + \sin\theta \cos\theta me^{i(\gamma-\mu- \alpha)} - r\sin^2\theta & re^{i(\beta-\alpha)}\sin\theta  \cos\theta + m e^{i(\delta-\mu-\alpha)}\cos^2 \theta - re^{i(\delta- \gamma)} \cos\theta \sin\theta \\
	re^{i(\alpha-\beta)}\sin\theta  \cos\theta + m e^{i(\gamma-\mu-\beta)}\sin^2 \theta - re^{i(\gamma- \delta)} \cos\theta \sin\theta  & r \sin^2\theta + \sin\theta \cos\theta me^{i(\delta-\mu- \beta)} - r \cos^2\theta\\
	\end{bmatrix}}.
	$$

Since the diagonal elements of the matrix $U^*AU \in \mathcal{B}(A)$,  $r \cos^2\theta + \sin\theta \cos\theta me^{i(\gamma-\mu- \alpha)} - r\sin^2\theta \in \mathcal{B}(A)$. If $m=|m|e^{i\zeta}$,  observe that 
\begin{equation*}
\begin{split}
r \cos^2\theta + me^{i(\gamma-\mu- \alpha)}\sin\theta \cos\theta  - r\sin^2\theta &= r \cos 2\theta + \frac{1}{2} |m| \sin 2 \theta e^{i(\gamma-\mu- \alpha+\zeta)}\\
      &= x +iy
\end{split}
\end{equation*}	
where
\begin{equation*}
\begin{split}
x &= r \cos 2\theta + \frac{1}{2} |m| \sin 2 \theta \cos(\gamma-\mu- \alpha+\zeta),\\
      y &= \frac{1}{2} |m| \sin 2 \theta \sin(\gamma-\mu- \alpha+\zeta).
\end{split}
\end{equation*}	
Then $$ (x - r \cos 2\theta)^2 + y^2 = \frac{|m|^2}{4}\sin^2(2\theta).$$
This is a family of circles. We obtain their union.
Put $2\theta = \phi$
\begin{equation}\label{4.1}
(x - r \cos \phi)^2 + y^2 = \frac{|m|^2}{4}\sin^2\phi, \quad 0\leq \phi \leq \pi.
\end{equation} 
Now, differentiating the above equation with respect to $\phi$, we get
$$(x - r \cos \phi)r= \frac{|m|^2}{4}\cos\phi.$$
Therefore $$ \cos\phi = \frac{xr}{r^2 + \frac{|m|^2}{4}}.$$
Substituting $\cos\phi$ in equation \ref{4.1} and eliminating $\phi$, we get
$$ \frac{x^2}{r^2 + \frac{|m|^2}{4}} + \frac{y^2}{\frac{|m|^2}{4}} = 1.$$
This is the equation of the ellipse with centre at $(0,0)$ and minor axis $|m|$, major axis $\sqrt{4r^2+|m|^2}$  and foci at $(r,0)$ and $(-r,0)$. Therefore we get $\mathcal{B}(A)$ is an ellipse with foci at $\lambda_1,\lambda_2$, minor axis $|m|$ and major axis $\sqrt{4r^2+|m|^2}$.
	\end{proof}

\begin{theorem}
Let $\mathcal{H}$ be an RKHS and $T$ be a bounded linear operator on $\mathcal{H}$. Then $\mathcal{B}(T) = W(T)$, and therefore $\mathcal{B}(T)$ is convex.
\end{theorem}
\begin{proof}
From Lemma \ref{sub}, it is clear that $\mathcal{B}(T) \subseteq W(T)$. So it is enough to show that $W(T) \subseteq \mathcal{B}(T)$.

Let $\alpha ,\beta \in W(T)$. Then $\alpha = \langle Tu,u\rangle$ and $\beta = \langle Tv,v \rangle$, where $u$ and $v$ are unit vectors in $\mathcal{H}$, i.e., $||u|| = ||v|| = 1$. Let $\mathcal{M}$ be the subspace of $\mathcal{H}$ spanned by $u$ and $v$ and $P$ be the orthogonal projection of $\mathcal{H}$ on $\mathcal{M}$, so that $Pu=u$ and $Pv=v$. Also note that for the operator $PTP$ on $\mathcal{M}$,
$$\langle PTPu,u \rangle = \langle Tu,u\rangle$$
and
$$\langle PTPv,v \rangle = \langle Tv,v\rangle .$$
Note that $PTP$ is an operator in a two-dimensional space. So by elliptic range theorem, $W(PTP)$ is an ellipse and $\alpha ,\beta \in W(PTP)$ . From Theorem \ref{range}, it is easy to observe that $\mathcal{B}(PTP) = W(PTP)$. Therefore  $\alpha ,\beta \in \mathcal{B}(PTP)$.
Now as $\alpha ,\beta \in \mathcal{B}(PTP)$, there are normalized reproducing kernels $\hat{k}_\alpha, \hat{k}_\beta \in \mathcal{M}$ such that
$$\alpha = \langle PAP \hat{k}_\alpha,\hat{k}_\alpha\rangle$$
and 
$$\beta = \langle PCP \hat{k}_\beta,\hat{k}_\beta\rangle$$
where $A$ and $C$ are operators on $\mathcal{H}$ that are unitarily equivalent to $T$. Therefore we have
$$\alpha = \langle PAP \hat{k}_\alpha,\hat{k}_\alpha\rangle = \langle A \hat{k}_\alpha,\hat{k}_\alpha\rangle \in \mathcal{B}(T)$$
and 
$$\beta = \langle PCP \hat{k}_\beta,\hat{k}_\beta\rangle = \langle C \hat{k}_\beta,\hat{k}_\beta\rangle  \in \mathcal{B}(T).$$
This implies  $\alpha,\beta \in \mathcal{B}(T).$ Hence $W(T) \subseteq \mathcal{B}(T)$.
\end{proof} 

\textbf{Declaration of competing interest}

There is no competing interest.\\

\textbf{Data availability}

No data was used for the research described in the article.\\

{\bf Acknowledgments.} The first author is supported by the Senior Research Fellowship (09/0239(13298)/2022-EM) of CSIR (Council of Scientific and Industrial Research, India). The second author is supported by the Researchers Supporting Project number(RSPD2024R1056), King Saud University, Riyadh, Saudi Arabia. The third author is supported by the Teachers Association for Research Excellence (TAR/2022/000063) of SERB (Science and Engineering Research Board, India).

\nocite{*}
\bibliographystyle{amsplain}
\bibliography{database}  
\end{document}